\documentclass{amsart}
\usepackage{amssymb} 
\usepackage{amsmath} 
\usepackage{mathrsfs} 
\usepackage{eufrak}
\usepackage{amscd}
\usepackage{amsbsy}
\usepackage{comment}
\usepackage[matrix,arrow]{xy}
\usepackage{hyperref}
\usepackage{mathtools}
\usepackage{esint}
\usepackage[letterpaper, margin=1in]{geometry}

\newcommand{\sphere}{{\mathbb S}}

\newcommand{\m}{{\mathfrak m}}

\newcommand{\Lra}{\Longrightarrow}

\newcommand{\inv}{^{-1}}

\newcommand{\ep}{\epsilon}

\newcommand{\tr}{\text{tr}}
\newcommand{\RNum}[1]{\uppercase\expandafter{\romannumeral #1\relax}}

\newcommand{\be}{\begin{equation}}
\newcommand{\ee}{\end{equation}}
\newcommand{\ba}{\begin{align*}}
\newcommand{\ea}{\end{align*}}

\usepackage{xcolor}

\newtheorem{lem}{Lemma}[section]
\newtheorem{prop}[lem]{Proposition}

\newtheorem{theorem}{Theorem}[section]

\newtheorem{remark}{Remark}[section]

\newtheorem{proposition}[lem]{Proposition}

\theoremstyle{definition}
\newtheorem{definition}{Definition}

\def\mod#1{{\ifmmode\text{\rm\ (mod~$#1$)}
\else\discretionary{}{}{\hbox{ }}\rm(mod~$#1$)\fi}}
\begin{document}

\bibliographystyle{amsplain}

\author{Albert Chau$^1$}
\address{Department of Mathematics,
The University of British Columbia, Room 121, 1984 Mathematics
Road, Vancouver, B.C., Canada V6T 1Z2} \email{chau@math.ubc.ca}

\thanks{$^1$Research
partially supported by NSERC grant no. \#327637-06}

\author[A. Martens]{Adam Martens}
\address{Department of Mathematics, The University of British Columbia, 1984 Mathematics Road, Vancouver, B.C.,  Canada V6T 1Z2.  Email: martens@math.ubc.ca. }

\title{On the Bartnik mass of non-negatively curved CMC spheres}

\maketitle

\vspace{-20pt}

\begin{abstract}

Let $g$ be a smooth Riemannian metric on $\mathbb{S}^2$ and $H>0$ a constant.  We establish an upper bound for the corresponding Bartnik mass $\m_B(\mathbb{S}^2, g, H)$ assuming that the Gauss curvature $K_g$ is non-negative.   Our upper bound approaches the Hawking mass  $\m_H(\mathbb{S}^2, g, H)$ when either $g$ becomes round or else $H\to 0^+$, the bound is zero for $H$ sufficiently large, and in any case the bound is not more than $r/2=\m_H(\mathbb{S}^2, g, 0)$.  We obtain upper bounds on $\m_B(\mathbb{S}^2, g, H)$ as well in the case when  $g$ is arbitrary and $H$ is sufficiently large depending on $g$.



\end{abstract}

\section{Introduction}

 Let $g$ be a smooth metric and $H$ a smooth function on $\mathbb{S}^2$.  The Bartnik mass \cite{Bartnik} of the triple $(\mathbb{S}^2, g, H)$ is defined as 
\begin{equation}\label{Bartnikmass}\m_B(\mathbb{S}^2, g, H):= \inf_{(M, \gamma)}\{\m_{ADM}(M, \gamma): (M, \gamma) \text{ is an admissible extension of }  (\mathbb{S}^2, g, H)  \}
\end{equation}
where $\m_{ADM}(M, \gamma)$ is the ADM mass \cite{ADM} and admissibility means that $M=[0,\infty)\times \mathbb{S}^2$ and $\gamma$ is  a smooth asymptotically flat Riemannian metric having non-negative scalar curvature such that $\gamma |_{\partial M} =g$, and $\partial M$ has mean curvature $H$ and is outer minimizing in $M$. The Bartnik mass associates a notion of quasi-local mass to a closed compact $3$ dimensional region $\Omega$ of an asymptotically flat space-like hypersurface in a Lorentzian $4$-manifold, where $\partial \Omega=\mathbb{S}^2$ and $g$ and $H$ are respectively the metric and mean curvature induced from the associated Riemannian metric $g_{\Omega}$ on $\Omega$.  Assuming the dominant energy condition and time symmetric setting, $g_{\Omega}$ will have non-negative scalar curvature, and if $(M, \gamma)$ is further required to smoothly extend $(\Omega, g_{\Omega})$ in \eqref{Bartnikmass} as in the original formulation in \cite{Bartnik}, then $\m_B(\mathbb{S}^2, g, H)$ will be non-negative by the positive mass theorem \cite{SY}.  

Another such notion of quasi-local mass associated to the triple $(\mathbb{S}^2, g, H)$ is the Hawking mass:
\begin{equation}\label{Hawkingmass}
\m_H(\mathbb{S}^2, g, H):= \sqrt{\frac{Area(\mathbb{S}^2, g)}{16\pi}}\left(1-\int_{\mathbb{S}^2} H^2\right).
\end{equation}
The proof of the Riemman Penrose inequality in \cite{HI} implies that the Bartnik mass is bounded below by the Hawking mass thus providing a positive lower bound for the Bartnik mass when $\m_H(\mathbb{S}^2, g, H)>0$.

In this article we will focus on the problem of bounding the Bartink mass from above.     In \cite{MS} Mantoulidis and Schoen showed that when the operator $-\Delta_g+ K_g$ has positive first eigenvalue and $H=0$, then the Bartnik mass actually attains the above lower bound as given by the Riemann Penrose inequality so that one has $$\m_B(\mathbb{S}^2, g, 0)=\m_H(\mathbb{S}^2, g, 0)=\sqrt{Area(\mathbb{S}^2)/16\pi}.$$  In \cite{CM} the authors of this paper extended the above result to the degenerate case when the operator $-\Delta_g+ K_g$ has zero first eigenvalue.  In  \cite{CMM, MWX, LS} various upper bounds for Bartnik mass were obtained under the assumption that the Gauss curvature $K_g$ is strictly positive and that $H$ is a positive constant.  In \cite{MX}, assuming the Gauss curvature $K_g$ is strictly positive and that $H$ is a positive function, Miao-Xie adapted the method in \cite{ST} and its variation in \cite{M} to construct admissible extensions with zero scalar curvature and ADM mass arbitrarily close to $\sqrt{Area(\mathbb{S}^2)/16\pi}$ from above, thus establishing the bound $$m_B(\mathbb{S}^2, g, H) \leq \sqrt{Area(\mathbb{S}^2)/16\pi}$$.

 In Theorem \ref{mainthm} below, assuming the Gauss curvature $K_g$ is non-negative and that $H$ is a positive constant we obtain an upper bound for the Bartnik mass where the bound is not more than $\sqrt{Area(\mathbb{S}^2)/16\pi}$ and is in fact equal to zero for all $H$ sufficiently large, while the bound approaches the Hawking mass as $g$ converges smoothly to a round metric and the Hakwing mass is non-negative.  Theorem \ref{mainthm} follows in part from Theorem \ref{mainthm2} which also implies an upper bound for the Bartnik mass for an arbitrary metric $g$ provided $H$ is sufficiently large depending on $g$, where again the upper bound equals zero for $H$ sufficiently large.


\begin{theorem}\label{mainthm}
Let $g$ a smooth Riemannian metric on $\mathbb{S}^2$ with $K_g\geq 0$ and $H>0$ a constant.  Then if $D=D(g, H)$ is the constant in Defnintion \ref{def1} we have



\begin{equation}\label{bound2}
\begin{split}
\m_B(\mathbb{S}^2, g, H) & \leq \max \left[\min \left(\frac{r\sqrt{1+D}}{2} \left[ 1-\frac{r^2H^2}{4(1+D)}\right], \,\,\,\, \frac{r}{2}\right), 0\right]
\end{split}
\end{equation}
where $r=\sqrt{Area(\mathbb{S}^2, g)/4\pi}$.  Moreover, $D(g, H)$ satisfies $D\to 0$ if either $H\to 0^+$ or $g$ converges smoothly to a round metric.
\end{theorem}

\begin{theorem}\label{mainthm2}
Let $g$ a smooth Riemannian metric on $\mathbb{S}^2$ and $H>0$ a constant.   Suppose there exists a $(g, H)$ admissible path $g(t)$ as in Definition \ref{def1}.  Then if $D=D(g, H)$ is the constant in Defnintion \ref{def1} we have

\begin{equation}\label{bound1}
\begin{split}
\m_B(\mathbb{S}^2, g, H) & \leq \max \left(\frac{r\sqrt{1+D}}{2} \left[ 1-\frac{r^2H^2}{4(1+D)}\right], \,\,\,\, 0\right)
\end{split}
\end{equation}
\end{theorem}
\begin{remark}The expression $\frac{r\sqrt{1+D}}{2} \left[ 1-\frac{r^2H^2}{4(1+D)}\right]$ approaches the Hawking mass $\m_H(\mathbb{S}^2, g, H)$ as $D\to 0$.
\end{remark}
\begin{remark}Given any smooth metric $g$, it is proven in Proposition \ref{p2} that a $(g, h)$ admissible path exists when either $K_g\geq 0$ or else when $H$ is sufficiently large depending on $g$.\end{remark}

\vspace{10pt}
\noindent $Acknowledgement$. The authors would like to thank Pengzi Miao for helpful comments and his interest in this work.

\vspace{10pt}
\section{Admissible paths and the definition of $D(g, H)$}
Given a smooth metric $g$ on $\mathbb{S}^2$ we define $$r_g:=\sqrt{Area(\mathbb{S}^2, g)/4\pi}$$ $$K_g:= \,\,\, \text{Gauss curvature of $g$}$$

 Now we define the following notions of admissible paths of metrics starting from $g$, and the corresponding constant $D(g, H)$ in Theorem \ref{mainthm}.

\begin{definition}\label{def0}

Given a  smooth metric $g$ on $\mathbb{S}^2$, a smooth path of metrics $\zeta=g(t)$ for $t\in [0, 1]$ is called 

$g$-$admissible$ if 
\begin{enumerate}
\item[a)] $g(0)=g$ and $g(1)$ is round (has constant curvature),
\item[b)] $\tr_{g} g'\equiv 0$ for all $t$.
\end{enumerate}


\end{definition}

\begin{definition}\label{def1}
Given a smooth metric $g$ on $\mathbb{S}^2$ and a constant $H>0$.  A $g$-admissible path $\zeta=g(t)$ is called $(g, H)$-$admissible$ if there exists a positive constant $C>0$ for which 
\begin{equation}\label{ghadmiss}\frac{4C^2}{H^2} K_{g(t)} (1+C\sqrt{t})-2t|g'|^2(1+C\sqrt{t})^{2}+C^2 >0\,\,\,\,\,\, \text{on} \,\,\, \sphere^2\times [0,1].\end{equation} 
Defining $C_{\zeta}$ to be the infemum of all such constants $C$, we define $D(g, H)$ as
\begin{equation}
D(g, H):=\inf_{\text{$(g, H)$-admissible paths $\zeta$}} C_{\zeta}.
\end{equation}
\end{definition}

The following proposition confirms the existence of admissible paths when $K_g \geq 0$ or else $H$ is sufficiently large depending on $g$, in which cases lower bounds for $D(g, H)$ are also provided.

\begin{prop}\label{p2} Let $g$ be a smooth metric and $H>0$. 
\begin{enumerate}

\item [1.] If $K_g\geq 0$ then $(g, H)$-admissible paths $\zeta=g(t)$ exist, one can be chosen to satsify
$$\frac{4K_{g(t)}+H^2}{ H^2}-\sqrt{2t^2|g(t) '|^2_{g(t)}}>0 \,\,\,\,\,\, \text{on} \,\,\, \sphere^2\times [0,1],$$ and for any such path we have the estimate:
\begin{equation}\label{Destimate1}D(g, H)\leq \max_{\mathbb{S}^2\times[0,1]}\frac{\sqrt{2t|g '|^2_{g}}}{\sqrt{\frac{4K_{g(t)}+H^2}{ H^2}}-\sqrt{2t^2|g '|^2_{g}}}\end{equation}
\item [2.] If $H$ is sufficiently large depending on $g$ then $(g, H)$-admissible paths exist.  In particular, a $g$-admissible path $\zeta=g(t)$ (independent of $H$) can always be chosen to satisfy \begin{equation}\label{E1}1-2t^2|g(t) '|^2_{g(t)}>0 \,\,\,\,\,\, \text{on} \,\,\, \sphere^2\times [0,1], \end{equation} and if we then choose any constant $C$ satisfying \begin{equation}\label{EE1}C>\max_{\mathbb{S}^2\times[0,1]}\frac{\sqrt{2t|g '|^2_{g}}}{1-\sqrt{2t^2|g '|^2_{g}}},\end{equation} then $\zeta$ will be $(g, H)$-admissible provided $H$ satisfies \begin{equation}\label{EEE1}H\geq \max_{\mathbb{S}^2\times[0,1]}\sqrt{\frac{4K_{g(t)}(1+C\sqrt{t})C^2}{2t|g'|^2(1+C\sqrt{t})^{2}-C^2}}\end{equation}
(where the definition of $C$ guarantees the positivity of the denominator) and in this case we have the estimate $D(g, H)\leq C$.

\end{enumerate}

\end{prop}

\begin{proof}

 Given any smooth metric $g$, we may write $g=e^{w(x)}g_*$ for a round metric $g_*$ with area $4\pi$.  The construction in \cite{MS} (Proposition 1.1 and Lemma 1.2) shows that given any smooth non-increasing function $\alpha(t):[0, 1]\to [0, 1]$ with $\alpha(0)=1$ and $\alpha(1)=0$, a function $a(t)$ and a family of diffeomorphisms $\xi_t:\mathbb{S}^2\to \mathbb{S}^2 $ can be chosen so that $\zeta=h(t):=\xi_t ^*(e^{2\alpha(t)w(x)+a(t)}g_*)$ is a $g$-admissible path.  In particular, choosing $\alpha(t)=1$ in a neighborhood of $t=0$ ensures that $a(t)$ must be constant, that $\xi_t$ remains constant by its construction, and thus  $h$ remains constant for all $t>0$ sufficiently small.  Fix such a path and consider the modified path of metrics $h_c(t)$ for $t\in [0, 1]$, defined as 
$$
h_c(t)=
\begin{cases}
g &0 < t\leq e^{-\frac{1}{c}}\\
h(c\log t+1) & e^{-\frac{1}{c}}\leq \ t\leq 1
\end{cases}
$$
First, we note that since $h(t)$ is assumed to be constant in a neighborhood of $t=0$ it follows that  $h_c(t)$ is smooth on $\mathbb{S}^2\times[0,1]$ for all $c>0$.  By the facts that $h(t)$ itself is admissible, and that $h_c(t)$ is just a re-parametrization of $h(t)$ we see that $h_c(t)$ is $g$ admissible as in Definition \ref{def0}.  We now proceed to complete the proof of each part of the Proposition separately.

$\mathbf{Part\, 1.}$  In this case we shrink $c$ if necessary so that $$0<c<\sqrt{ \frac{\min_{\mathbb{S}^2\times[0,1]}4K_{h_c(t)} +H^2}{ 2H^2\max_{\mathbb{S}^2\times[0,1]}  |h'(s)|^2_{h(s)}}}$$ which gives
\begin{equation}\label{admissible1}
\begin{split}
\max_{\mathbb{S}^2\times[0,1]} 2t^2 |h_c'(t)|_{h_c(t)}^2=\max_{\mathbb{S}^2\times [e^{-1/c}, 1]} 2t^2 |h_c'(t)|_{h_c (t)}^2&=\max_{\mathbb{S}^2\times[0,1]} c^2 2 |h'(s)|^2_{h(s)}<\frac{\min_{\mathbb{S}^2\times[0,1]}4K_{h_c(t)} +H^2}{H^2} \\
\end{split}
\end{equation}
and thus allows us to define 
\begin{equation}\label{C(g, H)}C:=\max_{\mathbb{S}^2\times[0,1]}\frac{\sqrt{2t|h_c '|^2_{h_c}}}{\sqrt{\frac{4K_{h_c(t)}+H^2}{ H^2}}-\sqrt{2t^2|h_c '|^2_{g}}}\end{equation}
which in turn implies 
\begin{align*}
0&\leq K_{h_c (t)}\frac{4C^2}{ H^2}-2t|h_c '|_{h_c}^2(1+C \sqrt{t})^{2}+C^2 \,\,\,\,\,\, \text{on} \,\,\, \mathbb{S}^2\times[0,1]\\
\end{align*}
 which is in fact a stronger inequality than \eqref{ghadmiss} thus making $h_c(t)$ a $(g, H)$ admissible path and also establishing the estimate for $D(g, H)$ in \eqref{Destimate1}.

$\mathbf{Part\, 2.}$ In this case, by the calculation in \eqref{admissible1} we see that we may shrink $c$ if necessary so that the $g$-admissible path $h_c(t)$ also satisfies \eqref{E1}.  If we take $C$ as in \eqref{EE1}, then 
$$-2t|h_c '|_{h_c}^2+(1+C \sqrt{t})^{-2}C^2 >0 \,\,\,\,\,\, \text{on} \,\,\, \mathbb{S}^2\times[0,1]$$ and it follows that when $H$ satsifies \eqref{EEE1} we have
$$K_{h_c (t)}(1+C \sqrt{t})^{-1}\frac{4C^2}{ H^2}-2t|h_c '|_{h_c}^2+(1+C \sqrt{t})^{-2}C^2 >0 \,\,\,\,\,\, \text{on} \,\,\, \mathbb{S}^2\times[0,1]$$ making $h_c(t)$ a $(g, H)$-admissible path and also establishing the stated estimate for $D(g, H)$.

\end{proof}

\begin{remark}\label{rem2}
The construction of the $g$-admissible path $\zeta=h(t)$ from \cite{MS} (Proposition 1.1 and Lemma 1.2) referred to in the proof of Proposition \ref{p2} implies that

\begin{enumerate}
\item [1] if $g$ is arbitrarily close to a fixed round metric on $\mathbb{S}^2$ (in each $C^k$ norm relative to a fixed metric), then $|g(t)'|^2$ will be arbitrarily close to $0$ on   $\mathbb{S}^2\times[0,1]$.
\item[2] if $K_g \geq 0$, then $K_{h(t)} >0 $ for all $t>0$.
\end{enumerate}

 It follows that if either $g$ is arbitrarily close to a fixed round metric on $\mathbb{S}^2$, or $H$ is arbitrarily large, the inequality
$$\frac{4C^2}{H^2} K_{g(t)} (1+C\sqrt{t})-2t|g'|^2(1+C\sqrt{t})^{2}+C^2 >0\,\,\,\,\,\, \text{on} \,\,\, \sphere^2\times [0,1]$$
will be satisfied for $C>0$ arbitrily small, and thus  $D(g, H)$ from definition \eqref{def1} will be arbitrarily close to $0$.

\end{remark}

\begin{prop}\label{p1} Let $g$ be a smooth metric and $H>0$ 
a constant such that there exists a $(g, H)$-admissible path.   Then given any $\epsilon>0$, a $(g, H)$-admissible path $\zeta=g(t)$ can be chosen so that $$C_{\zeta} \leq D(g, H)+\epsilon$$ and also $g(t)=g(1)$ for all $t\in [1-\theta, 1]$ for some $\theta >0$.   Moreover, 
\end{prop}
\begin{proof}
Let $\epsilon>0$ be given and consider a $(g, H)$-admissible path $\zeta=g(t)$ satisfying \eqref{ghadmiss} for some $C \leq D(g, H)+\epsilon$.  Namely,

\begin{equation}\label{ghadmissss}\frac{4C^2}{H^2} K_{g(t)} (1+C\sqrt{t})-2t|g'|^2(1+C\sqrt{t})^{2}+C^2>0 \,\,\,\,\,\, \text{on} \,\,\, \mathbb{S}^2\times[0,1]\end{equation}.

We will construct a family of $(g, H)$-admissible paths $g_{\theta}(t)$ satisfying $g_{\theta}(t)=g_{\theta}(1)$ for $t\in [1-\theta, 1]$ so that for all $\theta>0$ sufficiently small, \eqref{ghadmissss} is satisfied with the same constant $C$ but with $g(t)$ replaced with $g_{\theta}(t)$. For each $\theta\in (0, 1/3)$, we consider a smooth auxilary function $\sigma_\theta:[0,1]\to [0,1]$ satisfying
$$
\begin{cases}
&\sigma_\theta(t)=\frac{t}{1-2\theta}, \; \forall \; t\in [0,1-3\theta]\\
&\sigma_\theta(t)=1, \; \forall \; t\in [1-\theta,1]\\
&0\leq \sigma_\theta '(t)\leq \frac{1}{1-2\theta}, \; \forall \; t\in [0,1].
\end{cases}
$$
Such a function can be constructed by mollification as discussed in \cite{CMM}. Then the path $\zeta_\theta=g_\theta(t):=g(\sigma_\theta(t))$ satisfies $g_\theta(t)= g(1)$ for all $t\in [1-\theta, 1]$.  Moreover, we have

\begin{equation}
\begin{split}
2t|g_\theta'(t)|_{g_\theta(t)}^2&=2t\sigma_\theta'(t)^2 |g'(\sigma_\theta(t))|_{g(\sigma_\theta(t))}^2\\
&=  \frac{\sigma_\theta'(t)^2t}{\sigma_\theta(t)}(2\sigma_\theta(t) |g'(\sigma_\theta(t))|_{g(\sigma_\theta(t))}^2)\\&\leq \frac{t}{\sigma_\theta(t)(1-2\theta)^2}(2\sigma_\theta(t) |g'(\sigma_\theta(t))|_{g(\sigma_\theta(t))}^2)\\
\end{split}
\end{equation}
and thus letting $s=\sigma_\theta(t)$, we have
\begin{equation}\label{ghadmisssssss}\begin{split}&\frac{4C^2}{H^2} K_{g_\theta(t)} (1+C\sqrt{t})-2t|g'_\theta(t)|_{g(\sigma_\theta(t))}^2(1+C\sqrt{t})^{2}+C^2\\
&\geq\frac{4C^2}{H^2} K_{g(s)} (1+\frac{\sqrt{t}}{\sqrt{\sigma_\theta(t)}}C\sqrt{s})- \frac{t}{s(1-2\theta)^2}(2s |g'(s)|_{g(s)}^2)(1+\frac{\sqrt{t}}{\sqrt{\sigma_\theta(t)}}C\sqrt{s})+C^2\end{split}\end{equation} 
On the other hand, as $ t/\sigma_{\theta}(t)\to 1$ uniformly for $t\in [0, 1]$ as $\theta \to 0$ it follows from \eqref{ghadmissss} that for $\theta>0$ sufficiently small, the expression on the LHS of \eqref{ghadmisssssss} is positive for $t\in [0, 1]$, thus making $g_\theta(t)$ a $(g, H)-$admissible path.

\end{proof}

\section{collar metrics on $(\sphere^2\times [0,1], \gamma)$ and an extension result}

 The approach used in \cite{MS} to constructing admissible extensions of a given metric $g$ on $\sphere^2$ is to first construct a suitable collar metric $\gamma$ on $\sphere^2\times [0,1]$ which extends $g$, then glue this isometrically to a Riemannian manifold/extension $(\sphere^2\times (b,\infty), \widetilde{g})$ for $0<b<1$ so that the ADM mass of the extension is suitably controlled.  We will follow this approach here.

 We first collect some basic formulas we will use for our collar construction.  Let $t\mapsto g(t)$ be a $g$-admissible path of metrics. Let $M=(\sphere^2\times [0,1], \gamma)$ where $$\gamma=E(t) g(t)+\Phi(t)^2 dt^2.$$ Write $h(t)=E(t) g(t)$ to simplify our notation to $\gamma=h(t)+\Phi(t)^2 dt^2$. 

\noindent \textbf{Mean Curvature:} Fix any $t\in (0,1]$. The mean curvature of the sphere $\sphere^2\times \{t\}$ as a submanifold of $(\Sigma, \gamma)$ is $H_t=\tr_{h(t)}(\rho)$ where $\rho=\langle N,\RNum{2}\rangle_\gamma$. Here $N=-\frac{1}{\Phi(t)} \frac{\partial}{\partial t}$ is the unit normal and $\RNum{2}$ is the second fundamental form. To calculate this, let $E_1, E_2$ be a local coordinate frame on $\sphere^2\times \{t\}$. Then
\begin{align*}
\rho(t)_{ij}
&=\langle N,\RNum{2}(E_i,E_j)\rangle_\gamma
=\gamma_{ab} N^a ((\tilde{\nabla}_{E_i} E_j )^\perp)^b
=\gamma_{tt} N^t ((\tilde{\nabla}_{E_i} E_j )^\perp)^t
\\&=\gamma_{tt}\frac{1}{-\Phi(t)}\tilde{\Gamma}_{ij}^t
=\gamma_{tt}\frac{1}{-\Phi(t)}(-\frac{1}{2}\gamma^{tt}\gamma_{ij;t})
=\frac{1}{2\Phi(t)}h_{ij;t}.
\end{align*}
Since $h(t)=E(t) g(t)$, we have
$$
\dot{h}=E'(t) g(t)+E(t) \dot{g}(t) \; \Lra \; \tr_h \dot{h}=E\inv \tr_g \dot{h}=2E'(t)E\inv(t). 
$$
Note that we used $\tr_g \dot{g}\equiv 0$ in the above calculation. Therefore on $\Sigma_t:=\sphere^2\times \{t\}$, the foliating sphere at time $t$ we have
\begin{equation}\label{meancurvature}
H_t=\tr_{h(t)}\rho(t)=\frac{1}{\Phi(t)} \tr_h(\dot{h})=\frac{1}{\Phi(t)}E'(t)E\inv(t).
\end{equation}

\noindent \textbf{Scalar curvature:} The scalar curvature of $\gamma$ is given by the formula \cite{MS}:
\begin{equation}
R_\gamma=2K_{h(t)}+\Phi^{-2}\left[-\tr_h h'' -\frac{1}{4}(\tr_h h')^2+\frac{3}{4}|h'|^2_h+\frac{\partial_t \Phi}{\Phi}\tr_h h' \right].
\end{equation}
By some basic calculations, we have 
$$
\tr_h h'= 2E\inv E', \;\; \text{ and } \; \; |h'|_h^2=E^{-2}[2(E')^2+E^2 |g'|_g^2].
$$
Using this, we have
$$
-\frac{1}{4}(\tr_h h')^2+\frac{3}{4}|h'|^2_h=-E^{-2}(E')^2+\frac{3}{2}E^{-2}(E')^2+\frac{3}{4}|g'|_g^2 =\frac{1}{2}E^{-2}(E')^2+\frac{3}{4}|g'|_g^2.
$$
We know that 
$$
K_{h(t)}=E(t)\inv K_{g(t)} .
$$
and we also have 
$$
\tr_h h''= 2E\inv E'' +\tr_g g'' \; \text{ and } \; 0=[(\tr_g g')]'=\tr_g g''-|g'|_g^2 \; \Lra\; \tr_g g''=|g'|_g^2.
$$
This gives
\begin{equation}\label{scalar}
\begin{split}
R_\gamma&= 2K_{h(t)}+\Phi^{-2}\left[-\tr_h h''  -\frac{1}{4}(\tr_h h')^2+\frac{3}{4}|h'|^2_h+\frac{\partial_t \Phi}{\Phi}\tr_h h' \right]
\\&= 2E\inv K_{g(t)}+\Phi^{-2}\left[- 2E\inv E'' -\tr_g g'' +\frac{1}{2}E^{-2}(E')^2+\frac{3}{4}|g'|_g^2+2E\inv E'\frac{\partial_t \Phi}{\Phi}\right]\\&=
\Phi^{-2}\left[2E\inv K_{g(t)}\Phi(t)^2- 2E\inv E''-\frac{1}{4}|g'|_g^2+\frac{1}{2}E^{-2}(E')^2+2E\inv E'\frac{\partial_t \Phi}{\Phi}\right]\\&=
E\inv\Phi^{-2}\left[2 K_{g(t)}\Phi(t)^2- 2E''-\frac{1}{4}E|g'|_g^2+\frac{1}{2}E^{-1}(E')^2+2 E'\frac{\partial_t \Phi}{\Phi}\right].
\end{split}
\end{equation}

\noindent \textbf{Bartnik Mass:} Recall that the Hawking mass of a foliating sphere $\Sigma_t$ is given by 
$$
\m_H(\Sigma_t,h(t),H_t)=\sqrt{\frac{|\Sigma_t|}{16\pi}}\left(1-\frac{1}{16\pi}\int_{\Sigma_t}H_t^2 \right)
$$
 Now using the facts that (1) the path $g(t)$ has constant pointwise area form, and (2) scaling a metric by a constant factor scales the area by the same constant factor, we have $|\Sigma_t|=4\pi r_g^2 E(t)$ for each $t\in [0,1]$. From our mean curvature calculations above, we have 
$$
H_t=\frac{E'(t)E(t)\inv}{\Phi(t)}.
$$
Then
$$
\m_H(\Sigma_0)=\sqrt{\frac{|\Sigma_0|}{16\pi}}\left(1-\frac{1}{16\pi}\int_{\Sigma_0}H_0^2 \right)=\frac{r_g}{2}\left(1-\frac{r_g^2 H^2}{4} \right).
$$
Similarly, 
\begin{equation}\label{hawkingmass}
\m_H(\Sigma_1)=\sqrt{\frac{|\Sigma_1|}{16\pi}}\left(1-\frac{1}{16\pi}\int_{\Sigma_1}H_1^2 \right)=\frac{r_g\sqrt{E(1)}}{2}\left(1-\frac{r_g^2E'(1)^2 E(1)^{-1} }{4\Phi(1)^2} \right).
\end{equation}

After constructing suitable collars, we will use the following gluing/extension result from \cite{CMM} to produce admissible extensions with control on their ADM mass.

\begin{proposition}\label{prop2.1}[Proposition 2.1 in \cite{CMM}]
Consider a smooth metric $\gamma=f(t) g(t)+ dt^2$ on a cylinder $(\sphere^2\times [0, 1], \gamma)$.  Suppose
\begin{enumerate}
\item  $\gamma$ has positive scalar curvature 
\item  $g(t)=g_*$ (the standard round metric) and $f'(t) >0$ for all $s\in [a, 1]$ for some $0<a<1$
\item $\sphere^2\times \{1\}$ has positive mean curvature $H_1$
\item $m_H(\sphere^2\times\{1\},f(1) g(1), H_1) \geq 0$.
\end{enumerate}
Then given any $\epsilon>0$ there exists a smooth rotationally symetric metric $\widetilde{\gamma}$ on the manifold with boundary $\sphere^2\times [a, \infty)$ such that
\begin{enumerate}
\item [(a)] for sufficiently large $c>a$, $(\sphere^2\times (c, \infty), \widetilde{\gamma})$ is isometric to an exterior region of the  Schwarzschild manifold of mass $m:= m_H(\sphere^2\times\{1\},f(1) g(1), H_1)+\epsilon$:
$$(\sphere^2\times (2m, \infty), r^2 g_* + \frac{1}{1-\frac{2m}{r}} dr^2)$$
\item [(b)] $\widetilde{\gamma}=\gamma$ on $\sphere^2\times [a, (a+1)/2)$
\end{enumerate}
and thus in particular we have 

$$\m_B(\sphere^2\times\{0\},f(0) g(0), H_0)\leq \m_H(\sphere^2\times\{1\},f(1) g(1), H_1)$$

\end{proposition}

\section{Proof of Theorem \ref{mainthm2}}

Theorem \ref{mainthm2} will follow immediately from the following Theorem, the defninition of $D(g, H)$ in Definition \ref{def1}, Proposition \ref{p1}.

\begin{theorem}\label{theorem1}
Let $g$ be a Riemannian metric on $\mathbb{S}^2$ and $H>0$ a constant and  let $\zeta=g(t)$ be a $(g, H)$ admissible path, so that in particular,

\begin{equation}\label{ghadmissssss}\frac{4C^2}{H^2} K_{g(t)} (1+C\sqrt{t})-2t|g'|^2(1+C\sqrt{t})^{2}+C^2>0\,\,\,\,\,\, \text{on} \,\,\, \mathbb{S}^2\times[0,1]\end{equation}  for some $C>0$.  Suppose further that $g(t)=g(1)$ for $t<1$ sufficiently close to $1$.

Then 
$$
\m_B(\sphere^2,g,H) \leq \max\left( \frac{r_g\sqrt{1+C}}{2}\left[1-\frac{r_g^2H^2}{4(1+C)}\right], \,\, 0 \right).
$$
 
\end{theorem}

\begin{proof} Let $g, H$ and $g(t)$ be as in the Theorem.  Consider the smooth Riemannian 3 manifold with boundary $\widetilde{M}=(\sphere^2\times [0,1], \gamma)$ where 
$$
\gamma=(1+2H s)g\left(\frac{s^2}{4A^2}\right)+ds^2
$$
where $C$ is defined in the Theorem and $A=\frac{C}{2H}$.  We will show that $R_\gamma\geq 0$ on $\widetilde{M}$ and that the mean curvature of $\sphere^2\times\{t\}$ in $\widetilde{M}$ is positive and approaches $H$ as $t\to 0$.  We will do this after changing variables to $t=s^2/(4A^2)$ on $M=(\sphere^2\times (0,1])$ in which case $\gamma$ takes the form
 $$\gamma=E(t) g(t)+\Phi(t)^2 dt^2$$
where $E(t)=1+C \sqrt{t}$ and $\Phi(t)= A/\sqrt{t}$.  In these coordinates, we will show that $R_\gamma\geq 0$ on $M$ while the mean curvature $H_t$ of the foliating spheres in $M$ are positive and approach $H$ as $t\to 0$.  
\vspace{10pt}

\noindent \textbf{Claim 1 (mean curvature): $H_t>0$ is positive and approaches $H$ as $t\to 0$:} 

By \eqref{meancurvature} and our definitions of $E$ and $\Phi$ we have
$$
H_t=\frac{E'(t)}{\Phi(t)E(t)}=\frac{\frac{C}{2\sqrt{t}}}{(\frac{C}{2H\sqrt{t}})(1+C \sqrt{t})}=\frac{H}{1+C \sqrt{t}}
$$
Clearly $H_t> 0$ for all $t\in (0,1]$ and $H_t\to H$ uniformly as $t\searrow 0$.\\

\noindent \textbf{Claim 2 (scalar curvature): $R_{\gamma}\geq 0$ on $M$}.

 From our definition of $E$ and $\Phi$ we have
$$
-2E\inv E''+2E\inv E'\frac{\partial_t \Phi}{\Phi}=E\inv \left[-2 \left(-\frac{C}{4}t^{-3/2}\right)+2\left(\frac{1}{2}C t^{-1/2}\right)\left(\frac{-At^{-3/2}}{2At^{-1/2}}\right)\right]\equiv 0.
$$
and it follows from \eqref{scalar} that 
\begin{equation}\begin{split}
R_\gamma &\geq 2E\inv K_{g(t)}+ \Phi^{-2}\left[-\frac{1}{4}|g'|_g^2+\frac{1}{2}E^{-2}(E')^2\right]\\&= 2(1+C \sqrt{t})^{-1} K_{g(t)}+\frac{4H^2t}{C^2}\left[-\frac{1}{4}|g'|_g^2+\frac{1}{2}(1+C \sqrt{t})^{-2}\frac{C^2}{4t}\right]\\
&= (1+C \sqrt{t})^{-2} \frac{H^2}{2C^2}\left(\frac{4C^2}{H^2} K_{g(t)} (1+C\sqrt{t})-2t|g'|^2(1+C\sqrt{t})^{2}+C^2\right).\\
&\geq 0
\end{split}\end{equation}
on $\sphere^2\times [0,1]$  by \eqref{ghadmissssss}.\\

\noindent \textbf{Hawking mass of $(\Sigma_1:=\sphere^2\times\{1\})$}:
From \eqref{hawkingmass} and our definition of $E$ and $\Phi$ we have
$$
\m_H(\Sigma_1)=\sqrt{\frac{|\Sigma_1|}{16\pi}}\left(1-\frac{1}{16\pi}\int_{\Sigma_1}H_1^2 \right)=\frac{r_g\sqrt{E(1)}}{2}\left(1-\frac{r_g^2E'(1)^2 E(1)^{-1} }{4\Phi(1)^2} \right)=\frac{r_g\sqrt{1+C}}{2}\left(1-\frac{r_g^2H^2}{4(1+C)}\right).
$$

In conclusion, we have estabished that the metric $$
\gamma=(1+\frac{C}{2A} s)g\left(\frac{s^2}{4A^2}\right)+ds^2
$$
on $(\sphere^2\times [0,1], \gamma)$, introduced at the beginning of the proof, satisfies the hypothesis of Proposition \ref{prop2.1}  provided $H< \frac{2\sqrt{1+C}}{r_g}$, in which case we conclude 
$$\m_B(\sphere^2, g, H)\leq \m_H(\Sigma_1)=\frac{r_g\sqrt{1+C}}{2}\left(1-\frac{r_g^2H^2}{4(1+C)}\right).
$$
Now suppose $H\geq \frac{2\sqrt{1+C}}{r_g}$ and thus
$$ c:=\frac{H^2r_g^2}{4}-1\geq C >0$$
Then if  we consider the smooth Riemannian 3 manifold with boundary $\widetilde{M}=(\sphere^2\times [0,1], \gamma)$ where 
$$
\gamma=(1+2H s)g\left(\frac{s^2}{4A^2}\right)+ds^2
$$
and $A=\frac{c}{2H}$, the exact same proof as above show that $\gamma$ satisfies Claims 1 $\&$ 2 while our definition of $c$ gives $$\m_H(\Sigma_1)=\frac{r_g\sqrt{1+c}}{2}\left(1-\frac{r_g^2H^2}{4(1+c)}\right)=0,$$and again using Proposition \ref{prop2.1} we conclude $\m_B(\Sigma_1)\leq 0$. 

 This concludes the proof of the Theorem.

\end{proof}

\section{Proof of Theorem \ref{mainthm}}

\begin{proof}[proof of Theorem \ref{mainthm}] If $K_g\geq 0$ and $H>0$ is arbitrary, then the pair $(g, H)$ will satisfy the hypothesis of Theorem \ref{mainthm2} by part 1 of Proposition \ref{p1}.  Moreover, in this case the constant $D(g, H)\to 0$ as either $H\to 0^+$ or else $g$ converges smoothly to a round metric as pointed out in remark \ref{rem2}.  Theorem \ref{mainthm} then follows from Theorem \ref{theorem2} below.
\end{proof}

\begin{theorem}\label{theorem2}
Let $g$ be a Riemannian metric on $\mathbb{S}^2$ with $K_g\geq 0$ and let $H>0$ be a constant.  Then 
$$
\m_B(\sphere^2,g,H) \leq r/2$$

\end{theorem}

\begin{proof}

  Fix any $\epsilon \leq 1$ and consider the smooth Riemannian 3 manifold with boundary $M=(\sphere^2\times [0,1], \gamma)$ where 
$$
\gamma=E(t)g(t) +\Phi(t)dt^2
$$
where $g(t)$ is an $g$-admissible path, $E(t)=(1+\ep t)$ and
$$
\Phi(t) = \left\{
     \begin{array}{lr}
       A t+\frac{\ep}{H} & t\leq\frac{1}{4}\\
       \phi(t)& : \frac{1}{4}\leq t\leq \frac{1}{2}\\
       \frac{A}{4}+\frac{\ep}{H}+1  & : \frac{1}{2}\leq t
     \end{array}
   \right.
$$
for a constant $A$ to be chosen sufficiently large below.  As described in the proof of Proposition \ref{p2} the $g$-admissible path can be taken to have the form $\xi_t^* (e^{2\alpha(t)w(x)+ a(t)}g_*)$ where $\alpha(t):[0,1] \to [0, 1]$ is smooth and non-increasing, $g=w(x)g_*$ for a round metric $g_*$ with area $4\pi$, and $\xi_t$ is a family of diffeomorphisms so that in particular, choosing $\alpha(t)$ constant for $t$ sufficiently close to $1$ implies that  $g(t)$ is also constant for $t>0$ sufficently close to $t=1$.  Moreover, we can also bound $K_{g(t)}$ from below linearly in $t$ as $K_{g(t)}\geq ct$ on $\mathbb{S}^2\times[0,1]$  for some $c>0$. To see why this is true, recall the relationship of Gauss curvature of conformal metrics: 
$$
K_{g(t)}\circ \xi_t=K_{e^{2\alpha(t)w(x)-2a(t)}g_*}=e^{2a(t)} K_{e^{2\alpha(t)w(x)}g_*}=e^{2a(t)-2\alpha(t)w(x)} (1-\alpha(t)\Delta_* w).
$$
We know that $\Delta_* w \leq 1$ on $\sphere$ as $K_{g(0)}\geq 0$. Letting $B=\inf_{t,x} e^{2a(t)-2\alpha(t)w(x)} $ gives
$$
K_{g(t)}\circ \xi_t\geq B (1-\alpha(t))=B(1-\alpha'(0)t +O(t^2))\gtrsim B\alpha'(0)t .
$$

Now \eqref{scalar} implies that in order to ensure $R_\gamma\geq 0$ it suffices to show
$$
ct\, \Phi(t)^2 -2C+\frac{1}{4}E\inv (E')^2+4E' \frac{\partial_t \Phi}{\Phi}\geq ct\, \Phi(t)^2-C+\ep \frac{\partial_t \Phi}{\Phi}\geq 0
$$
where $C:=\max_{\mathbb{S}^2\times[0,1]} \frac{1}{4}|g'|^2_{g}$.  To this end, let $\delta=\min\left\{\frac{1}{4}, \frac{\ep}{2C}\right\}$. Then since $\delta\leq 1/4$ we have $\Phi(t)\leq \Phi(\delta)$ for all $t\in [0,\delta]$.  If we choose  $A\geq 2C/H$ we may then estimate for all $t\in [0,\delta]$ as
$$
\frac{\partial_t \Phi(t)}{\Phi(t)}\geq \frac{A}{\Phi(\delta)}=\frac{A}{A\delta+\frac{\ep}{H}}\geq \frac{A}{A\frac{\ep}{2C}+\frac{\ep}{H}}=\left(\frac{C}{\ep}\right) \frac{\frac{A}{2}+\frac{C}{2}}{\frac{A}{2}+\frac{C}{H}}=\frac{C}{\ep}.
$$
thus implying $R_\gamma\geq 0$ for $t\in [0,\delta]$.
Now if we further choose $A\geq\sqrt{ \frac{C}{c\delta^3}}$, then for $t\in [\delta, 1]$, 
$$
ct \,\Phi(t)^2=ct \left(A\delta+\frac{\ep}{H}\right)^2\geq c\delta^3 A^2\geq c\delta^3\sqrt{ \frac{C}{c\delta^3}}^2 =C
$$
thus implying $R_\gamma\geq 0$ for $t\in [\delta,1]$. So whenever $A$ is sufficiently large (as described), we have $R_\gamma\geq 0$ on  $[0,1]$.\\

 By \eqref{meancurvature}  we have that the mean curvature $H_t$ of the sphere $\sphere^2\times{t}$ in $M$ is given by 
$$
H_t=\frac{E'(t)E(t)\inv}{\Phi(t)}
$$
which is clearly positive by our choice of $E(t)$ and $\Phi(t)$, and is equal to $\frac{\epsilon (1+\epsilon t)\inv}{(At+\frac{\epsilon}{H})}$ for $t\leq 1/4$ and thus approaches $H$ as $t\to 0$.  

Moreover, by \eqref{hawkingmass} and our choice of $E(t)$ and $\Phi(t)$ the Hawking mass of $\Sigma_1=\sphere^2\times {1}$ is given by
\begin{equation}\label{last}
\begin{split}
\m_H(\Sigma_1)&=\frac{r_g\sqrt{E(1)}}{2}\left(1-\frac{r_g^2E'(1)^2 E(1)^{-1} }{4\Phi(1)^2} \right)\\&=\frac{r_g\sqrt{(1+\epsilon)}}{2}\left(1-\frac{r_g^2\epsilon^2 }{4(1+\epsilon)(A/4+\epsilon/H +1)^2} \right)
\end{split}
\end{equation}

Noting that in the above construction, $\epsilon$ could have been chosen arbitrarily small while $A$ could have be chosen arbitrarily large, it follows from \eqref{last} and Proposition \ref{prop2.1} that $\m_B(\sphere^2, g, H)\leq r_g/2$.

\end{proof}

\end{document}